\def\IK{{\Bbb K}}

\def\IC{\Bbb C} 
 
\def\ID{{\Bbb D}}

\def\mA{\mathcal{A}}
\def\mB{\mathcal{B}}
\def\mC{\mathcal{C}} 
\def\mE{\mathcal{E}}
\def\mP{\mathcal{P}}
\def\mS{\mathcal{S}}

\def\zbar{{\overline{z}}}

\def\wbar{{\overline{w}}}

\documentclass[12pt]{article} 

\usepackage[psamsfonts]{amssymb} 
\usepackage{amsfonts,amsmath}

\usepackage{epsfig,multicol}

\newtheorem{theorem}{Theorem}%[section] 
\newtheorem{lemma}{Lemma}%[section] 
\newtheorem{definition}{Definition}[section] 

\title{$L^p$-Extremal Teichm\"uller mappings between Riemann surfaces are diffeomorphisms}
\author{Gaven Martin \& Cong Yao\thanks{\tiny
Work of both authors partially supported by the New Zealand Marsden Fund. C. Yao is partially supported by  the Young Scientist Program of the Ministry of Science and Technology of China (No. 2021YFA1002200), the National Natural Science Foundation of China (No. 12401096, No. 12101362), the Natural Science Foundation of Shandong Province (No. ZR2024QA035, No. ZR2022YQ01), the Fundamental Research Funds for the Central Universities.
\newline
Institute for Advanced Study, 
Massey University,  Auckland,
New Zealand.
\newline
email: g.j.martin@massey.ac.nz  \newline
Research Center for Mathematics and Interdisciplinary Sciences, Shandong University, 266237, Qingdao and Frontiers Science Center for Nonlinear Expectations, Ministry of Education, P. R. China. \newline
email:  c.yao@sdu.edu.cn 
\newline
\newline
{\bf MSC Subject.}  30C62 31A05 49J10  } }
\date{}
\begin{document} 
\maketitle
\begin{abstract}
We consider minimisers in the homotopy class of a homeomorphism $f_0:R\to S$ between analytically finite Riemann surfaces with minimal $L^p$- conformal energy
\[ \mathsf{E}_p(f:R,S)=\int_R \IK^p(z,f)\;  d\sigma_R(z). \]
The problem was first raised by Ahlfors in his celebrated proof of Teichm\"uller's theorem---the case $p=\infty$,  but the existence, topological regularity and analytic regularity of these $L^p$ minimisers remained unknown for all $1<p<\infty$. Ahlfors established weak existence for $p\geq 2$. Here we prove that for all p, $1\leq p<\infty$, such minimisers exist, are unique and are diffeomorphisms. They are quasiconformal but not diffeomorphic at $p=\infty$.
\end{abstract}
\newpage
\section{Introduction}
\subsection{Teichm\"uller's theorem and the $L^p$ minimisation problems.}
Let $R$, $S$ be analytically finite hyperbolic Riemann surfaces and $f:R\to S$ be a homeomorphism. Teichm\"uller's theorem \cite{T} states that in the homotopy class of $f$, denoted $[f]$, there is a unique extremal mapping which has smallest maximal distortion
\begin{equation}
\IK(z,f)=\frac{\|Df(z)\|^2}{J(z,f)} = \frac{1+|\mu_f|^2}{1-|\mu_f|^2},\quad \mu_f(z)=\frac{f_\zbar(z)}{f_z(z)}
 \end{equation} 
In the first complete proof of Teichm\"uller's Theorem \cite{Ah},   Ahlfors initially sought, for $1\leq p< \infty$ a variational solution for 
\begin{equation}
\min_{g\in[f]} \; \int_R \IK^p(z,g) \, d\sigma_R,
\end{equation}
minimising the mean distortion in $L^p(R,d\sigma_R)$ (in fact he slightly modified this with a ``convergence factor'' to control the universal covering lift) and then let $p\to\infty$ to identify the maximal distortion. He saw that that the inverse of the $L^p$-minimal mapping,  should it even exist,  would satisfy the equation
\begin{equation}\label{1}
\Phi_p=\IK^{p-1}(w,h)h_w\overline{h_\wbar}\, d\sigma_R(h).
\end{equation}
Here $\Phi_p$ is a holomorphic differential -  we now call the Ahlfors-Hopf differential. Ahlfors then considered the ``inverse'' problem for $h$,  for $p\geq 2$ Ahlfors showed that $h\in W^{1,2}(R)$,  the Sobolev space of mappings with square integral first derivatives. Ahlfors then let $p\to\infty$ and found that the sequence $h_p$ converges strongly in $W^{1,2}(R)$ to a mapping of constant distortion $\mu = k \frac{\bar \Phi}{|\Phi|}$, with $\Phi$ a holomorphic quadratic differential,  now known as a Teichm\"uller mapping.  Here the constant $k<1$, hence we find a quasiconformal homeomorphism. However, although the limit function is (miraculously!) a quasiconformal homeomorphism, the regularity of the $L^p$ minimisers, or their ``inverses'', was unknown. As a consequence of Riemann-Roch,  the quadratic differential has $3g-3$ zeros on $R$ (closed of genus $g$). As $\mu$ cannot be smooth at such points, the extremal Teichm\"uller mapping cannot be a diffeomorphism. 

In this article we complete this circle by proving that in fact the $L^p$ variational problem Ahlfors considered has a unique diffeomorphic minimiser in any homotopy class of homeomorphism.  This work relies heavily on our earlier work on the related {\em exponentially integrable} distortion problem as there the problem of topological regularity goes away, see \cite[\S 20.4]{AIM}. We should point out that the case $p=1$ has long been settled, though only recently observed.  The minimiser is the inverse of the harmonic diffeomorphism $h=S\to R$, as seen at (\ref{1}) since we obtain the classical Hopf equation $h_z \overline{h_{\zbar}} d\sigma_R$ is holomorpic,  see Schoen-Yau \cite{SY}.  Our main result is then the following.

\begin{theorem}\label{MainTheorem} Let $(R,\sigma_R)$ and $S$ be analytically finite (finite area) Riemann surfaces.  Let $[f]$ be a homotopy class of homeomorphism between $R$ and $S$.  Then for $1\leq p<\infty$ the problem of finding 
\begin{equation}\label{2}
\min_{g\in \mathcal{F}[f]} \; \int_R \IK^p(z,g) \, d\sigma_R,
\end{equation}
admits a unique homeomorphic minimiser. This minimiser is also a diffeomorphism. Here the class 
\[ \mathcal{F}[f] = \{\mbox{mappings $g:R\to S$ of \underline{finite distortion} and $g\in[f]$}\} \]
\end{theorem}
 Of course if $R,S$ are closed,  then this diffeomorphism will be quasiconformal -- this is likely more general but we will pursue this elsewhere. We could have required $\mathcal{F}[f]$ to consist of diffeomorphisms with the same result. The class of mappings of finite distortion is about as weak as possible to formulate the problem,  so we soon give those definitions. We will see that as $p\to\infty$ we obtain convergence in $W^{1,2}$ to the Teichm\"uller mapping, and as $p\to 1$ we obtain convergence to a diffeomorphism whose inverse $h$ is the harmonic diffeomorphism, recovering Schoen-Yau's theorem. Finally,  each minimiser has an inverse which is harmonic in a metric determined by it's distortion,  that is solves the tension equation
 \[ h_{w\wbar}+(\log \rho)_z(h) h_wh_\wbar = 0. \]
Here, the metric term $\rho$ is the hyperbolic metric multiplied by the distortion $\IK^{p-1}(z,f)$, $f^{-1}=h$.
 
\subsection{Notation and precise formulation.} 
\begin{definition}
Let $\Omega\subset R$ be any domain on the surface and let $f\in W_{loc}^{1,1}(\Omega)$ be a Sobolev function on $\Omega$. Let $\|Df(z)\|^2$ be the Hilbert-Schmidt norm of $Df$, and $J(z,f)$ be the Jacobian determinant of $f$. We say $f$ is a finite distortion function, if $J(z,f)\in L_{loc}^1(\Omega)$ and there is a measurable function $\IK(z,f)$, finite almost everywhere, such that
\[
\|Df(z)\|^2\leq\IK(z,f)J(z,f).
\]
The function 
\[ \IK(z,f)=\frac{\|Df(z)\|^2}{J(z,f)}\]
 is called the distortion function of $f$.
\end{definition}

Now we can formulate Ahlfors' $L^p$ problem. Let $1<p<\infty$, $R$, $S$ be analytically finite surfaces, we define the mean $L^p$ distortion
\[
\mathsf{E}_p(f):=\int_R\IK^p(z,f) \; d\sigma_R(z),
\]
where $f:R\to S$ is a finite distortion homeomorphism. Given a ``barrier'' $f_0:R\to S$,   a  finite distortion homeomorphism  
with finite mean $L^p$ distortion, i.e. $\mathsf{E}_p(f_0)<\infty$, we wish to find the minimisers of $\mathsf{E}_p$ in the homotopy class of $f_0$. In fact the barrier can always be chosen to be quasiconformal for analytically finite Riemann surfaces.

\subsection{Addressing the lack of $W^{1,2}$-regularity.}
In \cite{MY3} we studied exponential minimisers in a given homotopy class of homeomorphisms between surfaces. The main reason to turn to the exponential problem is that the $L^p$ minimisers are not automatically homeomorphisms, whereas mappings of exponential distortion are \cite[\S20]{AIM}. In fact, if  $\mathsf{E}_p(f)<\infty$, then $f\in W^{1,q}(R,S)$ for all $1\leq q\leq \frac{2p}{p+1}<2$, since
\[
\int_R\|Df(z)\|^\frac{2p}{p+1}d\sigma_R(z)\leq\Big(\int_R\IK^p(z,f)d\sigma_R(z)\Big)^\frac{1}{p+1}\Big(\int_RJ(z,f)d\sigma_R(z)\Big)^\frac{p}{p+1}<\infty,
\]
see \cite{IMO, MY2}. However, (almost) $W^{1,2}$-regularity is needed to prove that they are homeomorphisms \cite{GV, IM}; otherwise we can only find a minimiser in the sense of an enlarged space,   formulated in the same way as in \cite[\S 5]{MY2}. Precisely, we say a function $f$ is in the space $\mathsf{F}_p$ and it is a pseudo-inverse of a function $h\in \mathsf{H}_p$ if they satisfy the following conditions.
\begin{itemize}
\item $f\in W^{1,\frac{2p}{p+1}}(R,S)$.
\item $h:S\to R$ is monotone (the preimage of any point is a connected closed set), homotopic to $h_0=f_0^{-1}$, and
\[
\int_S\IK^p(w,h)d\sigma_R(h)<\infty.
\]
\item There is a measurable set $X\subset R$ such that $|R-X|=0$ (full measure), $h\circ f(z)=z$ for every $z\in X$, and $J(w,h)=0$ for almost every $w\in S\setminus f(X)$.
\end{itemize}
We remark that because a minimiser is not necessarily a homeomorphism, this leads to the failure of our previous methods employed in \cite[\S 3.6]{MY3} applied in the exponential case. In this article, we will approximate the $L^p$ minimisers by exponential minimisers and prove that these approximations converge to diffeomorphisms to solve this problem.

\subsection{Existence of diffeomorphic minimisers.}
The basic analytic machinery used here is discussed in detail in \cite[\S1.4]{MY3}.  Our main result of \cite{MY3} was the following theorem.
\begin{theorem}(\cite[Theorem 16]{MY3})
Let $R,S$ be analytically finite Riemann surfaces, $f_0:R\to S$ be a finite distortion homeomorphism, and
\[
\mE_p(f_0)=\int_R\exp(p\IK(z,f_0)) \; d\sigma_R(z)<\infty,\quad0<p<\infty.
\]
Then, in the homotopy class of $f_0$, there exists a unique homeomorphic minimiser $f$ of $\mE_p$. Moreover, the minimiser $f$ is a diffeomorphism.
\end{theorem}

We remark that the same proof as given in \cite{MY3} is easily seen to work in the case $\exp(q\IK^p(z,f))$ for $q>0$, $p\geq1$, as noted in that article. But the result fails for the sub-exponential cases. This is a result of the lack of $W^{1,2}$-regularity just as we discussed above.

We approximate $L^p$ minimisers by the sequence of the $\exp(q\IK^p)$ minimisers with $q\to0$. We will show this is a minimising sequence and the limit is the unique diffeomorphic minimiser of the $L^p$- mean distortion.  

\subsection{Limiting regimes.}
Here we notice two extremal case when $p$ varies in the $L^p$ problems, for $1< p<\infty$.
\begin{itemize}
\item As $p\to1$,
\begin{equation}\label{lim1}
\int_R\IK^p(z,f) \; d\sigma_R(z)\to\int_R\IK(z,f) \; d\sigma_R(z)=\int_S\|Dh(w)\|^2 \; d\sigma_S(w),
\end{equation}
where $h:S\to R$ is the inverse mapping of $f$.
\item As $p\to\infty$,
\begin{equation}\label{lim2}
\left(\frac{1}{|R|}\int_R\IK^p(z,f) \; d\sigma_R(z)\right)^\frac{1}{p}\to\|\IK(z,f)\|_{L^\infty(R)},
\end{equation}
where $|R|$ is the area of $R$.
\end{itemize}

Then (\ref{lim1}) and (\ref{lim2}) above indicate that with the diffeomorphic $L^p$ minimisers, we can recover the harmonic mappings by letting $p\to0$ and the Teichm\"uller mappings by letting $p\to\infty$, also as discussed.

\begin{theorem}\label{limit}
Let $R,S$ be analytically finite Riemann surfaces, $f_0:R\to S$ be a quasiconformal homeomorphism. Let $f_p:R\to S$ be the diffeomorphic minimiser of $\mathsf{E}_p$ in the homotopy class of $f_0$, $h_p=f_p^{-1}$. Let $\Phi_p$ be the Ahlfors-Hopf differentials. Then 
\begin{itemize}
\item As $p\to1$, there is a harmonic mapping $h_0$ and its inverse $f_0$. $h_p\to h_0$ strongly in $W^{1,2}(S,R)$ and uniformly in $S$. $\Phi_p\to\Phi_1$ uniformly with all derivatives;
\item As $p\to p_0$, $1<p_0<\infty$, $f_p\to f_{p_0}$, $h_p\to h_{p_0}$, $\Phi_p\to\Phi_{p_0}$, all uniformly with all derivatives;
\item As $p\to\infty$, there is a Teichm\"uller mapping $f_\infty$ and its inverse $h_\infty$, which satisfies
\[
\mu_{h_\infty}=k\frac{\overline{\Phi_\infty}}{|\Phi_\infty|},
\]
where $k$ is a constant and $\Phi_\infty$ is a holomorphic quadratic differential in $S$. The Ahlfors-Hopf differentials $\Phi_p\to\Phi_\infty$ locally uniformly with all derivatives in $S\setminus Z$, where $Z$ is the zero set of $\Phi_\infty$. Moreover, $h_p\to h_\infty$ uniformly in $S$, weakly in $W^{1,2}(S,R)$, strongly in $W^{1,s}(S,R)$ for all $1\leq s<2$, and locally uniformly with all derivatives in $S\setminus Z$; $f_p\to f_\infty$ strongly in $W^{1,s}(R,S)$ for all $1\leq s<2$ and locally uniformly with all derivatives in $R\setminus h_\infty(Z)$.
\end{itemize}
\end{theorem}
The convergence of $h_p\to h_\infty$ cannot be too strong in $S$, and certainly not strongly in $W^{2,2+\epsilon}(S)$. To see that,  we observe that for any $1\leq p<\infty$ at a stationary point of the distortion we must have (dropping the $p$)
\[ (|\mu|^2)_w=0, \quad h_{\wbar w} h_w=h_{ww} h_\wbar, \quad (h_{w})_\wbar =\mu_h  (h_{w})_w \]
showing that $h_w$ is quasiregular,  hence open and discrete, near any stationary point.  If $S$ is a closed surface of genus $g$, any quadratic differential has $3g-3$ zeros, and for the Ahlfors-Hopf differential this has the implication that $|\mu_h|=0$,  since $J(w,h)>0$, $h_w\neq 0$ as a diffeomorphism. Hence at these zeroes $\Delta h_p=|\mu_p|=(h_p)_\wbar=0$.  While $|\mu_\infty|\equiv k\leq 1$, and $k\neq 0$ unless $R$ and $S$ are conformally equivalent (so all $h_p$ are the same conformal map). Thus we cannot have  $(h_p)_\wbar$ converging uniformly for any subsequence, near the zeros of the Ahlfors-Hopf differential as $p\to\infty$.   

\bigskip

\section{Proof of Theorem \ref{MainTheorem}}
\subsection{The $\exp(q\IK^p(z,f))$ problem.}
We first consider the problems
\[
\mE_{p,q}(f):=\int_R\exp(q\IK^p(z,f))\; d\sigma_R(z).
\]
We approximate by the linear combinations
\[
\mE_N(f):=\sum_{n=1}^N\int_\ID\frac{q^n\IK^{pn}(z,f)}{n!}\;  d\sigma_R(z)=\sum_{n=1}^N\int_\ID\frac{q^n\IK^{pn}(w,h)}{n!}\;  d\sigma_R(h).
\]
For each $N$ (as per \cite[§5 \& Theorem 5.2]{MY2}) we obtain a holomorphic Ahlfors-Hopf differential from the inner variational Euler-Lagrange equations.  With the inverse function $h_N=f_N^{-1}$, where $f_N$ is the minimiser of $\mE_N$ in the enlarged space we have
\[
\Phi_N=\sum_{n=1}^N\frac{q^n\IK(w,h_N)^{pn-1}}{(n-1)!}(h_N)_w\overline{(h_N)_\wbar}\,  d\sigma_R(h_N).
\]
As $N\to\infty$, $f_N\to f$, $h_N\to h=f^{-1}$, and the holomorphic sequence $\Phi_N$ converges to
\[
\Phi=\exp(q\IK^p(w,h))\IK^{p-1}(w,h)h_w\overline{h_\wbar}\, d\sigma_R(h),
\]
where $\Phi$ is also a holomorphic Ahlfors-Hopf differential. The same method as in the exponential case applies here and gives a unique diffeomorphic minimiser for the $\mE_{p,q}(f)$ problem, see \cite{MY3}.

\begin{theorem}
Let $R$, $S$ be analytically finite Riemann surfaces, $f_0 : R\to S$ be a finite distortion homeomorphism, and
\[
\mE_{p,q}(f_0)=\int_R \exp(q\IK^p(z,f_0))\;  d\sigma_R(z)<\infty,
\]
where $1<p<\infty$, $0<q<\infty$. Then, in the homotopy class of $f_0$, there exists a unique homeomorphic minimiser $f$ of $\mE_{p,q}$. Moreover, the minimiser $f$ is a diffeomorphism.
\end{theorem}

We remark this process actually works for all $\Psi(\IK(z,f))$ problem where $\Phi(t)$ is a convex increasing function with exponential growth, i.e. $\Psi(t)\geq\exp(pt)$ for some $p>0$, since they can be approximated by convex increasing polynomials $P_n(t)$ locally uniformly.

\subsection{Convergence of $f_q$ and $h_q$.}
We now fix a $1<p<\infty$. For each $q>0$, there is a diffeomorphic $\mE_{p,q}$ minimiser $f_q$ and it has an inverse $h_q$. As $q\to0$, we find that there is a limit function in the enlarged space. Precisely, $f_q\to f$ weakly in $W^{1,\frac{2p}{p+1}}(R,S)$, $h_q\to h$ weakly in $W^{1,2}(S,R)$, $f$ is the pseudo-inverse of $h$, and by the polyconvexity (see \cite{AIMO}),
\[
\frac{1}{|R|}\int_R\IK^p(z,f)\, d\sigma_R(z)\leq\liminf_{q\to0}\frac{1}{q}\log\left(\frac{1}{|R|}\int_R\exp(q\IK^p(z,f_q))\, d\sigma_R(z)\right).
\]

\subsection{Convergence of derivatives.}
As in \cite{MY3}, for the Riemann surfaces $R$ and $S$, there are covering mappings $\pi_R:\ID\to R$ and $\pi_S:\ID\to S$, where $\ID$ is the Pioncar\'e disk on the complex plane $\IC$, with the hyperbolic metric $\eta(z)=\frac{1}{(1-|z|^2)^2}$. Fix the fundamental polyhedrons $\mP,\mP'\subset\ID$, each $f_q$ has a lift $\tilde{f_q}:\mP\to\mP'$ which is a diffeomorphism. In the following we will abuse the notations of $f_q$ and their lifts, and write $\IK_f$, $\IK_h$ for $\IK(z,f)$, $\IK(w,h)$.

From standard results in the the calculus of variations (Fatou's theorem), a minimiser $f_q$ of $\mathcal{E}_{p,q}$ satisfies the inner variational equation (see \cite{IMO,MY2,MY3})
\begin{equation}\label{ELequation}
\int_\mP p\exp(q\IK_{f_q}^p)\IK_{f_q}^{p-1}\frac{2\overline{\mu_{f_q}}}{1-|\mu_{f_q}|^2}\eta\varphi_\zbar dz=\frac{1}{q}\int_\mP\left((\exp(q\IK_{f_q}^p)-e^q)\eta+\sigma^q\right)\varphi_zdz,
\end{equation}
where $\varphi\in C_0^\infty(\mP)$, and $\sigma^q=-\mC^*|_\mP\left((\exp(q\IK_{f_q}^p)-e^q)\eta_z\right)$. This gives a term $F^q\in C^\infty(\mP)$ such that
\[
F^q_z=p\exp(q\IK_{f_q}^p)\IK_{f_q}^{p-1}\frac{2\overline{\mu_{f_q}}}{1-|\mu_{f_q}|^2}\eta,\quad F^q_\zbar=\frac{1}{q}[(\exp(q\IK_{f_q}^p)-e^q)\eta+\sigma^q].
\]
We then get a functional relation between $F^q_\zbar$ and $F^q_z$.
\begin{equation}\label{2.1}
F^q_\zbar=\mA_q\left(\frac{|F^q_z|}{\eta}\right)\eta+\frac{\sigma^q}{q},
\end{equation}
where $\mA_q$ is the inverse of the function
\begin{equation}\label{2.2}
a_q(s)=p(sq+e^q)r^{p-1}\sqrt{r^2-1},
\end{equation}
where we have written
\begin{equation}\label{2.3}
r=r(s)=\left(\frac{1}{q}\log(sq+e^q)\right)^\frac{1}{p}\geq1.
\end{equation}
Note as $q\to0$ we obtain $r\to(s+1)^\frac{1}{p}$, and
\[
a(s)=p(s+1)^\frac{p-1}{p}\sqrt{(s+1)^\frac{2}{p}-1},
\]
which is exactly as in the $L^p$ case, see \cite{MY2}. From (\ref{2.2}), (\ref{2.3}) we can compute
\[
a_q'(s)=pqr^{p-1}\sqrt{r^2-1}+\frac{pr-(p-1)r^{-1}}{\sqrt{r^2-1}}:=qb(r)+c(r).
\]
Here $b(r)=pr^{p-1}\sqrt{r^2-1}\geq0$, and
\[
c'(r)=\frac{1}{(r^2-1)^\frac{3}{2}}[(1-p)r^{-2}+p-2].
\]
Thus for $r\geq1$, $c(r)$ attains its minimal value at $r=\sqrt{\frac{p-1}{p-2}}$ if $p>2$; or at $r=\infty$ if $1<p\leq2$. This gives
\[
K_p:=\min_{r\geq1}c(r)=\begin{cases}
p,&1<p\leq2\\
2\sqrt{p-1},&p>2
\end{cases}
\]
We then get
\[
a_q'(s)=qb(r)+c(r)\geq c(r)\geq K_p>1.
\]
In particular, the estimate is uniform and we achieve the $L^p$ case as $q\to0$. Let $k_p=1/K_p$, we have therefore established the following lemma.
\begin{lemma}
For equation (\ref{2.1}), there is a $k_p\in[0,1)$ s.t.
\[
\mathcal{A}_q'(t)\leq k_p,
\]
where $k_p$ is independent of $q>0$.
\end{lemma}
Now recall
\[
\frac{\sigma^q}{q}=-\mC^*|_\mP\left(\frac{1}{q}(\exp(q\IK_{f_q}^p)-e^q)\eta_z\right)\in W_{loc}^{1,1}(\mP)\subset L_{loc}^2(\mP),
\]
where the norms are also uniform as $q\to0$. Now since $\mA_q(0)=0$, we have
\[
\mA_q\left(\frac{|F^q_z|}{\eta}\right)\eta\leq\sup_{t\in\left[0,\frac{|F^q_z|}{\eta}\right]}\mA_q'(t)\frac{|F^q_z|}{\eta}\cdot\eta\leq k_p|F^q_z|.
\]
Thus we can write (\ref{2.1}) as
\[
F^q_\zbar=\nu^qF^q_z+\tilde{\sigma}^q,
\]
where
\[
|\nu^q|=\frac{\mA_q\left(\frac{|F^q_z|}{\eta}\right)\eta}{|F^q_z|}\leq k_p,\quad\tilde{\sigma}^q=\frac{\sigma^q}{q}\in L_{loc}^2(\mP).
\]
Then there is a uniform $L^2_{loc}(\mP)$ norm for $F^q_\zbar$:
\[
F^q_\zbar=(I-\nu^q\mS|_\mP)^{-1}\left(\tilde{\sigma}^q\right)\in L^2_{loc}(\mP).
\]
This gives a uniform $W_{loc}^{1,2}(\mP)$ norm for $F^q$. Also, the term
\[
F^q_\zbar=\frac{1}{q}[(\exp(q\IK_{f_q}^p)-e^q)\eta+\sigma^q]
\]
in turn gives $\frac{1}{q}\exp(q\IK_f^p)\in L^2_{loc}(\mP)$, also uniformly as $q\to0$. In particular, we have $\tilde{\sigma}^q\in W^{1,2}_{loc}(\mP)$, thus $F^q\in W_{loc}^{1,s}(\mP)$ uniformly for all $1\leq s<\infty$.

We now rewrite (\ref{2.1}) as follows. Let $\mB_q$ be the inverse of $a_q^2$, i.e. $\mB_q(t^2)=\mA_q(t)$. Then
\[
F^q_\zbar=\mB_q\left(\frac{|F^q_z|^2}{\eta^2}\right)\eta+\tilde{\sigma^q}.
\]
We differentiate both side of the equation by $x$ and get
\[
(F^q_x)_\zbar =\mathcal{B}_q^\prime\left(\frac{|F^q_z|^2}{\eta^2}\right)\frac{\overline{F^q_z}}{\eta}(F^q_x)_z+\mathcal{B}_q^\prime\left(\frac{|F^q_z|^2}{\eta^2}\right)\frac{F^q_z}{\eta}\overline{(F^q_x)_z}+\phi^q(z),
\]
where
\[
\phi^q(z)=\tilde{\sigma}^q_x+\mathcal{B}_q\left(\frac{|F^q_z|^2}{\eta^2}\right)\eta_x-2\mathcal{B}_q^\prime\left(\frac{|F^q_z|^2}{\eta^2}\right)\frac{|F^q_z|^2}{\eta^2}\eta_x\in L_{loc}^2(\mP).
\]
Also,
\[
\mathcal{B}_q^\prime\left(\frac{|F^q_z|^2}{\eta^2}\right)\frac{|F^q_z|}{\eta}+\mathcal{B}_q^\prime\left(\frac{|F^q_z|^2}{\eta^2}\right)\frac{|F^q_z|}{\eta}\leq\mathcal{A}_q^\prime\left(\frac{|F^q_z|}{\eta}\right)\leq k_p.
\]
This gives a uniform $W_{loc}^{1,2}(\mP)$ bound for $F_x^q$. The same manipulations work for $F^q_y$ and then we can keep doing this inductively to see that the $F^q$ are uniformly bounded in $C_{loc}^k(\mP)$ space for any $k\geq1$ . In particular, we get a limit function $F$ which is smooth and, since $f_q\to f$ in $W^{1,1}(\mP)$,
\[
F_z=\IK_f^{p-1}\frac{2p\overline{\mu_f}}{1-|\mu_f|^2}\eta,\quad F_\zbar=(\IK_f^p-1)\eta-\mC^*|_\mP\left((\IK_f^p-1)\eta_z\right).
\]
Now the smoothness of $F$ implies that $\mu_f$ is smooth, and so $f$ is a diffeomorphism from $\mP$ to $\mP'$. Transforming this it back to the surfaces we get that $f:R\to S$ is a diffeomorphism. Since $R$ and $S$ are analytically finite surfaces, there is no degeneracy phenomenon of the cracks from the boundary that we met with in the disk case in \cite{MY2,MY5}. In fact, for any point $z\in\partial\mP$, we may find another fundamental polyhedron $\tilde{\mP}\subset\ID$ such that $z\in\tilde{\mP}^\circ$, and then $f$ is a diffeomorphism at $z$.

\subsection{Uniqueness; convergence of Ahlfors-Hopf differentials.}
We consider the Ahlfors-Hopf differentials
\[
\Phi_q=\exp(q\IK_{h_q}^p)\IK_{h_q}^{p-1}(h_q)_w\overline{(h_q)_\wbar}\, d\sigma_R(h_q).
\]
By the Riemann-Roch Theorem, we have $\Phi_q\to\Phi$ uniformly, where $\Phi$ is also a holomorphic quadratic differential on $S$ (see \cite[\S 2.2]{MY3}). Since $h_q\to h$ in $W^{1,2}(S,R)$, we have
\[
\Phi=\IK_h^{p-1}h_w\overline{h_\wbar}\, d\sigma_R(h).
\]
With the holomorphic quadratic differential $\Phi$, we can prove that $h$ is a unique minimiser of the inverse $L^p$ problem, see \cite{MY1}. This completes the proof of Theorem \ref{MainTheorem}.

\bigskip

\section{Equations for the diffeomorphic minimisers.}
We have already noted that the minimiser satisfies a tension equation depending on its own distortion, we prove this in \S3.3 below. However there are other important equations which are useful for our future research about how the associated  Teichm\"uller and moduli spaces of the domain move under the variation of $p$.  We will see in a moment that this is via smooth path with a conserved quantity.
\subsection{Inner variational equations and Ahlfors-Hopf differentials.}
Let $q\to0$ in (\ref{ELequation}) or compute directly with an inner variation, we obtain the inner variational equation for an $L^p$ minimiser;
\begin{equation}\label{InnerVariation}
\int_\mP\IK_f^{p-1}\frac{2\overline{\mu_f}}{1-|\mu_f|^2}\eta\varphi_\zbar dz=\int_\mP(\IK_f^p-1)(\eta\varphi)_zdz,
\end{equation}
where $\mP\subset\ID$ is a fundamental polyhedron of the surface $R$ and $\varphi\in C_0^\infty(\mP)$. Similarly, the inner variation of the inverse mapping gives an equation
\begin{equation}\label{hInnerVariation}
\int_{\mP'}\IK_h^{p-1}h_w\overline{h_\wbar}\eta\varphi_\wbar dw,
\end{equation}
where $\mP'\subset\ID$ is a fundamental polyhedron of the surface $S$ and $\varphi\in C_0^\infty(\mP')$. By Weyl's lemma, the function
\[
\IK_h^{p-1}h_w\overline{h_\wbar}\eta(h)
\]
is holomorphic in $\mP'$. Map the functions back to the surfaces to get the following holomorphic Ahlfors-Hopf differential
\begin{equation}\label{AHdifferential}
\Phi=\IK_h^{p-1}h_w\overline{h_\wbar}\, d\sigma_R(h),
\end{equation}
just as in \S2.4.
\subsection{Equations for the Beltrami coefficient $\mu$.}
Now since $f$ is a diffeomorphism, (\ref{InnerVariation}) can be read as
\[
\left(p\IK_f^{p-1}\frac{2\overline{\mu_f}}{1-|\mu_f|^2}\eta\right)_\zbar=\left(\IK_f^p-1\right)_z\eta.\]
As (\ref{InnerVariation}) holds for any fundamental polyhedron $\mP$, this holds everywhere in $\ID$. This leads to an equation for $\mu=\mu_f$
\begin{equation}\label{muequation}
\gamma(|\mu|)\mu_z=\alpha(|\mu|)\mu_\zbar-\beta(|\mu|)\overline{\mu_\zbar}-\psi,
\end{equation}
where
\begin{align*}
\gamma(t)&=1+(4p-3)t^2+(4p^2-4p-1)t^4-t^6,\\
\alpha(t)&=(1+t^2)^2(1-t^2),\\
\beta(t)&=2(p-1)(1+2pt^2+t^4),\\
\psi&=(1-|\mu|^4)(1+(2p-1)|\mu|^2)\mu\frac{\eta_z}{\eta}+(1+|\mu|^2)^2(1-|\mu|^2)|\mu|^2\frac{\eta_\zbar}{\eta}.
\end{align*}
We consider
\[
A(t):=\frac{\alpha(t)t+\beta(t)t^2}{\gamma(t)}=\frac{t+(2p-1)t^2+t^3+t^4}{1+t+(2p-1)t^2+t^3}
\]
Then
\[
\frac{1+A^2}{1-A^2}\leq\frac{C}{1-A}=\frac{C[1+t+(2p-1)t^2+t^3]}{1-t^4}\leq\frac{C}{1-t},
\]
and
\[
\frac{|\psi|}{\gamma(|\mu|)}\leq\frac{(1-|\mu|^4)|\mu|}{1-|\mu|+(2p-1)|\mu|^2-|\mu|^3}\left|\frac{\eta_z}{\eta}\right|\leq C\left|\frac{\eta_z}{\eta}\right|.
\]
Thus we can rewrite (\ref{muequation}) as
\[
\mu_z=\nu\mu_\zbar+\phi,
\]
where
\[
\frac{1+|\nu|^2}{1-|\nu|^2}\leq C_1\IK(z,f),\quad|\phi|\leq C_2\left|\frac{\eta_z}{\eta}\right|,
\]
Note here $C_1=C_1(p)$ is bounded as $p\to1$ but has growth $C_1(p)\sim p$ as $p\to\infty$, and $C_2$ is independent of $p\in(1,\infty)$.

\subsection{Tension equation.}
Starting from (\ref{hInnerVariation}), we may compute as follows
\[
0=\left(\IK_h^{p-1}h_w\overline{h_\wbar}\right)_\wbar=\left(\IK_h^{p-1}\right)_\wbar h_w\overline{h_\wbar}+\IK_h^{p-1}h_{w\wbar}\overline{h_\wbar}+\IK_h^{p-1}h_w\overline{h_{w\wbar}}.
\]
This leads to
\[
h_{w\wbar}+(\log\rho)_z(h)h_wh_\wbar=0,
\]
where
\[
\rho(z)=\IK^{p-1}(z,f)\eta(z).
\]
This is the {\it tension equation} for the harmonic mapping with respect to the metric $\rho(z)|dz|^2$, see \cite{M}. That is to say, $h$ is a harmonic mapping from $\ID$ to $(\ID,\rho)$.

\section{Limit regimes; proof of Theorem \ref{limit}}
\subsection{$p\to p_0$, $1\leq p_0<\infty$ case.}
As $p\to p_0$, we note the boundedness of the sequence
\[
\int_R\IK^{p_0}(z,{f_p})\; d\sigma_R(z)\leq\int_R\IK^p(z,{f_p})\; d\sigma_R(z)<\infty,
\]
which gives a uniform $W^{1,2p_0/(p_0+1)}(R,S)$ bound of $f_p$. This also implies the uniform bound of
\[
\int_S\|Dh_p(w)\|^2d\sigma_S(w)\leq\int_S\IK^p(w,h_p)J(w,h_p)d\sigma_S(w)=\int_R\IK^p(z,{f_p})d\sigma_R(z).
\]
This gives a uniform $W^{1,2}(S,R)$ bound of $h_p$. This gives the weak convergence of $f_p\to f_{p_0}$ in $W^{1,2p_0/(p_0+1)}(R,S)$ and the weak convergence of $h_p\to h_{p_0}$ in $W^{1,2}(S,R)$, and this gives the uniform convergence of $h_p$, also see \cite{AIM,GV,IM}. Moreover, by the polyconvexity (see \cite[\S12]{AIMO}) and the uniqueness of the minimisers, we have
\begin{equation}\label{3.1}
\int_S\IK^{p_0}(w,h_{p_0})J(w,h_{p_0})\; d\sigma_S(w)=\lim_{p\to p_0}\int_S\IK^p(w,h_p)J(w,h_p)\; d\sigma_S(w).
\end{equation}
With equation (\ref{3.1}), we may apply the Radon-Riesz property of finite distortion integrals to get that $h_p\to h_{p_0}$ strongly in $W^{1,2}(S)$, see \cite{MY4}.\\

For $1<p_0<\infty$ case, we may start from the inner variational equation (\ref{InnerVariation}) to have a sequence of functions $F^p$ such that
\[
F^p_z=\IK_f^{p-1}\frac{2\overline{\mu_f}}{1-|\mu_f|^2}\eta,\quad F_\zbar=(\IK_f^p-1)\eta+\sigma^p.
\]
They satisfy the equations
\[
F^p_\zbar=\mA_p\left(\frac{|F^p_z|}{\eta}\right)\eta+\sigma^p.
\]
where $\mA_p'\leq k<1$ is uniform with $p$. The same argument as in \S2.3 gives the convergence of higher-order derivatives as $p\to p_0$. The convergence of the Ahlfors-Hopf differentials $\Phi_p$ follows from the strong convergence of $h$.

\subsection{$p\to+\infty$ case.}
A similar argument as in the last section gives the boundedness of the energies, which gives the weak convergence of $f_p\to f_\infty$ in $W^{1,s}(R)$, $1\leq s<2$ and that of $h_p\to h_\infty$ in $W^{1,2}(S)$. By Vitali's theorem, they can be improved to strong convergence whence we have the pointwise almost everywhere convergence which we will see below.

To see the convergence of higher order derivatives, we lift the functions to the Poincar\'e disk and consider
\begin{equation}\label{Phi}
\Phi=\IK_h^{p-1}h_w\overline{h_\wbar}\eta(h).
\end{equation}
where $h=h_p$, but we dropped the subscripts for notational ease. Let $Z=\{w\in\ID:\Phi(w)=0\}$. As $\Phi$ a holomorphic mapping, $Z$ is a discrete subset of $\ID$. For any $w\in\ID\setminus Z$, there is a neighbourhood $w\in D\subset\ID\setminus Z$. In $D$, we can choose $\Psi=\sqrt{\Phi}$, where $\sqrt{\Phi}$ is any well-defined branch. In $\Omega:=f(D)$, we can define
\[
g:=\Psi(f)f_z,
\]
and
\[
\nu:=-\frac{g}{\overline{g}}\mu_f=-\frac{\Phi(f)f_z^2}{|\Phi(f)f_z^2|}\mu_f=\frac{\IK_f^{p-1}\overline{f_zf_\zbar}f_z^2\eta/J_f^2}{\IK_f^{p-1}|f_z^3f_\zbar|\eta/J_f^2}\mu_f=|\mu_f|.
\]
Thus $\nu$ is a real positive function. We may also compute that
\[
g_z=\Psi'(f)f_z^2+\Psi(f)f_{zz},
\]
\[
g_\zbar=\Psi'(f)f_zf_\zbar+\Psi(f)f_{z\zbar}.
\]
Thus
\[
g\cdot (\mu_f)_z=\frac{g}{f_z}(f_{z\zbar}-\mu_f f_{zz})=\Psi(f)f_{z\zbar}-\Psi(f)\mu_f f_{zz}=g_\zbar-\mu_f g_z,
\]
\[
\nu_z=-\left(\frac{g}{\overline{g}}\mu_f\right)_z=-\left(\frac{g(\mu_f)_z}{\overline{g}}+\left(\frac{g}{\overline{g}}\right)_z\mu_f\right)=-\frac{g_\zbar+\nu\overline{g_\zbar}}{\overline{g}}.
\]
Now (\ref{Phi}) gives
\[
\Phi(f)=-\IK_f^{p-1}\frac{\overline{f_zf_\zbar}}{J_f^2}\eta=-\left(\frac{1+\nu^2}{1-\nu^2}\right)^{p-1}\frac{\overline{\mu_f}}{f_z^2(1-\nu^2)^2}\eta.
\]
Thus
\[
|g|^2=|\Phi(f)f_z^2|=\left(\frac{1+\nu^2}{1-\nu^2}\right)^{p-1}\frac{\nu}{(1-\nu^2)^2}\eta,
\]
\[
\log|g|^2=(p-1)\log\left(\frac{1+\nu^2}{1-\nu^2}\right)+\log\nu-2\log(1-\nu^2)+\log\eta.
\]
Put $z$-derivatives on both sides,
\[
\frac{g_z\overline{g}+g\overline{g_\zbar}}{|g|^2}=\left(\frac{4(p-1)\nu}{1-\nu^4}+\frac{1}{\nu}+\frac{4\nu}{1-\nu^2}\right)\nu_z+\frac{\eta_z}{\eta}.
\]
Put the value of $\nu_z$ into it,
\[
\xi g_z+\overline{g_\zbar}=-\frac{1+4p\nu^2+3\nu^4}{(1-\nu^4)\nu}(g_\zbar+\nu \overline{g_\zbar})+\overline{g}\frac{\eta_z}{\eta},
\]
where $\xi=\frac{\overline{g}}{g}$. We write
\[
\alpha=1+4p\nu^2+3\nu^4,\quad\beta=\nu(1-\nu^4).
\]
Then the last equation reads as
\[
\alpha g_\zbar+(\nu\alpha+\beta)\overline{g_\zbar}+\beta\xi g_z=\beta\overline{g}\frac{\eta_z}{\eta}.
\]
We solve for $g_\zbar$ in the last equation, and obtain
\begin{equation}
\Gamma g_\zbar=A\overline{\xi g_z}-B\xi g_z+\psi,
\end{equation}
where
\[
\Gamma=\alpha^2-(\nu\alpha+\beta)^2,\quad A=(\nu\alpha+\beta)\beta,\quad B=\alpha\beta,
\]
\[
\psi=\alpha\beta\overline{g}\frac{\eta_z}{\eta}-(\nu\alpha+\beta)\beta g\frac{\overline{\eta_z}}{\eta}.
\]
Note
\[
\frac{A+B}{\Gamma}=\frac{\beta}{(1-\nu)\alpha-\beta}=\frac{\nu(1+\nu)(1+\nu^2)}{4p\nu^2+(1-\nu)(1-\nu^2-2\nu^3)}\leq\frac{C}{\sqrt{p}}.
\]
\[
\frac{|\psi|}{\Gamma}\leq\frac{A+B}{\Gamma}\frac{|\eta_z|}{\eta}|g|\leq\frac{C}{\sqrt{p}}\frac{|\eta_z|}{\eta}|g|.
\]

We now put back the subscripts. From this discussion, we see that the sequence $g_p$ is uniformly bounded in $W^{1,2}_{loc}(\Omega)$ as $p\to+\infty$. In fact it follows inductively that $g_p\to g$ for some smooth $g$ with all derivatives. In particular, the limit function satisfies the equation
\begin{equation}\label{gequation}
g_\zbar=0.
\end{equation}
As $f_p$ are uniformly bounded in $W^{1,s}_{loc}(\Omega)$ for some $1<s<2$, $h_p$ are uniformly bounded in $W^{1,2}_{loc}(\Omega)$, and $\Phi_p$ are uniformly bounded in $L^1(\Omega)$, there are limit functions $f$, $h$ and $\Phi$, where $f_p\to f$ weakly in $W^{1,s}_{loc}(\Omega)$, $h_p\to h$ weakly in $W^{1,2}_{loc}(\Omega)$, $\Phi_p\to\Phi$, $\Psi_p\to\Psi$ locally uniformly, and we have $g=\Psi(f)f_z$, $\Psi^2=\Phi$, and the uniform convergence of $g_p$ also implies that $f_p\to f$ and $h_p\to h$ locally uniformly in $\Omega$ with all derivatives. For the sequence
\[
\nu_p=-\frac{g_p}{\overline{g_p}}\mu_{f_p},
\]
we have already seen that each $\nu_p$ is real and $|\nu_p|\leq1$, so they converge to some $\nu$ weakly in $L^q(\Omega)$ for any $1<q<\infty$, and 
\begin{equation}\label{nuequation}
\nu=-\frac{g}{\overline{g}}\mu_f
\end{equation}
is also real. Also, it follows from (\ref{gequation}) that
\[
\nu_z=-\frac{g_\zbar+\nu\overline{g_\zbar}}{\overline{g}}=0.
\]
This proves that $\nu=k$ is a real constant. Thus (\ref{nuequation}) reads as
\[
k=\nu(h)=-\frac{g}{\overline{g}}\mu_f(h)=-\frac{\Psi(f)f_\zbar}{\overline{\Psi(f)f_z}}(h)=\frac{\Psi}{\overline{\Psi}}\overline{\mu_h}.
\]
Thus
\begin{equation}\label{TeichEquaiton}
\mu_h=k\frac{\overline{\Psi}}{\Psi}=k\frac{\overline{\Phi}}{|\Phi|}.
\end{equation}
Note this works at every point $w\in\ID\setminus Z$, so (\ref{TeichEquaiton}) holds almost everywhere in $\ID$. This proves $h$ is a Teichm\"uller mapping.

\medskip

Finally, we map the functions back to the surfaces with the covering maps $\pi_R:\ID\to R$, and $\pi_S:\ID\to S$ and the proof of Theorem \ref{limit} is completed.

\section{The variation of $p$--conformal energy}

In this section,  we again consider the $L^p$--extremal problem between surfaces, and for the convenience of the reader summarise what we have achieved. Then we show that $\mathcal{E}_p(f)$ varies smoothly with $p$, $1\leq p\leq\infty$ and investigate some of the consequences. Recall;
\[
\mathcal{E}_p(f):=\int_R\IK^p(z,f)\; d\sigma_R(z),
\]
where $f:R\to S$ is a homeomorphism in the homotopy class of a given quasiconformal homeomorphism $f_0:R\to S$. Each $1\leq p\leq\infty$ gives a unique diffeomorphic minimiser $f_p$. We now consider how the minimal energy $\mathcal{E}_p(f_p)$ changes with $p$. An easy observation is the following.
\begin{theorem} The function
\[
F(p):=\mathcal{E}_p(f_p)
\]
is an increasing function continuously differentiable function for $p\in(1,\infty)$ and
\[
F'(p)=\int_R\IK^p(z,f_p)\log\IK(z,f_p)\; d\sigma_R(z)
\]
\end{theorem}
Note this mapping will be strictly increasing unless there is a conformal mapping in the homotopy class of $f_0$, in which case the conformal mapping is the minimiser for all $p$, and $F(p)$ is constant.

\noindent{\bf Proof.}
If we fix a minimiser $f_p$. The integral
\[
q\mapsto\int_R\IK^q(z,f_p)\; d\sigma_R(z)
\]
is a convex increasing function. Let $q>p$, since $\mathcal{E}_q(f_q)$ attains the minimal energy of $\mathcal{E}_q$,
\[
\frac{\mathcal{E}_q(f_q)-\mathcal{E}_p(f_p)}{q-p}\leq\frac{\mathcal{E}_q(f_p)-\mathcal{E}_p(f_p)}{q-p}\leq \int_R\IK^q(z,f_p)\log\IK(z,f_p)\; d\sigma_R(z)
\]
If $q<p$, we have
\[
\frac{\mathcal{E}_p(f_p)-\mathcal{E}_q(f_q)}{p-q}\leq\frac{\mathcal{E}_p(f_q)-\mathcal{E}_q(f_q)}{p-q}\leq \int_R\IK^p(z,f_q)\log\IK(z,f_q)\; d\sigma_R(z)
\]
In either case, let $q\to p$ we get 
\[
\lim_{q\to p}\frac{\mathcal{E}_q(f_q)-\mathcal{E}_p(f_p)}{p-q}\leq \int_R\IK^p(z,f_p)\log\IK(z,f_p)\; d\sigma_R(z)
\]
Similarly, for $q>p$ we have
\[
\frac{\mathcal{E}_q(f_q)-\mathcal{E}_p(f_p)}{q-p}\geq\frac{\mathcal{E}_q(f_q)-\mathcal{E}_p(f_q)}{q-p}\geq \int_R\IK^p(z,f_q)\log\IK(z,f_q)\; d\sigma_R(z)
\]
and for $q<p$,
\[
\frac{\mathcal{E}_p(f_p)-\mathcal{E}_q(f_q)}{p-q}\geq\frac{\mathcal{E}_p(f_p)-\mathcal{E}_q(f_p)}{p-q}\geq \int_R\IK^q(z,f_p)\log\IK(z,f_p)\; d\sigma_R(z)
\]
Thus
\[
\lim_{q\to p}\frac{\mathcal{E}_q(f_q)-\mathcal{E}_p(f_p)}{p-q}\geq \int_R\IK^p(z,f_p)\log\IK(z,f_p)\; d\sigma_R(z)
\]
We conclude
\[
\lim_{q\to p}\frac{\mathcal{E}_q(f_q)-\mathcal{E}_p(f_p)}{p-q}= \int_R\IK^p(z,f_p)\log\IK(z,f_p)\; d\sigma_R(z)
\]
This means that the function
\[
F(p):=\mathcal{E}_p(f_p)
\]
is differentiable and
\[
F'(p)=\int_R\IK^p(z,f_p)\log\IK(z,f_p)\; d\sigma_R(z)
\]
In particular, this is also continuous in $p$, so we have that $F(p)$ is a $C^1$ function of $p$.\hfill $\Box$

\bigskip

\noindent{\bf Remark}: As we have proved that as a function of $p$, if $F(p)$ is differentiable at $p$, we in fact should have
\begin{align*}
F'(p)=\int_R\IK^p(z,f_p)\log\IK(z,f_p)d\sigma_R(z)+p\int_R\IK^{p-1}(z,f_p)\left(\frac{\partial}{\partial p}\IK(z,f_p)\right)d\sigma_R(z)
\end{align*}
Thus
\[
p\int_R\IK^{p-1}(z,f_p)\left(\frac{\partial}{\partial p}\IK(z,f_p)\right)\; d\sigma_R(z)=0 
\]
and this integral represents a conserved quantity.  One can use the $\mu$-equations at (\ref{muequation}) to explore this, but it is quite complicated and so we shall leave this for another time.

\end{document}